\documentclass[11pt,a4paper,
fleqn,reqno]{amsart}                 
\usepackage[T1]{fontenc}                   
\usepackage[utf8]{inputenc}                 
\usepackage[english]{babel}    
 \usepackage{csquotes}
\usepackage{graphicx}                      %
\usepackage[font=small]{quoting}    
\usepackage{float}
\usepackage{caption}  
\usepackage[top=1.4in,bottom=1.6in,left=1.in,right=1.in]{geometry}                
\usepackage{verbatim}
\usepackage{color}
\usepackage{xcolor}
\usepackage{picture}
\usepackage{amsmath,amsthm,amsfonts,amssymb}
\usepackage{comment}

\usepackage{hyperref}
\hypersetup{colorlinks,linkcolor={blue}, citecolor={blue},urlcolor={red}}  
\usepackage{enumitem}
\usepackage[url=false,doi=false,isbn=false, giveninits=true,style=numeric,maxbibnames=99]{biblatex}
\newtheorem{thm}{Theorem}[section] 
\newtheorem{lem}{Lemma}[section]  
\newtheorem{cor}[thm]{Corollary}

\newtheorem{defn}{Definition}[section] 
\theoremstyle{definition}

\theoremstyle{remark}

\numberwithin{equation}{section}

\def\S{\Sigma} 
\def\n{\nabla}

\def\a{\alpha}

\def\n{\nabla}

\def\a{\alpha}

\def\G{\Gamma}
\def\l{\lambda}
\def\s{\sigma}

\def\ov{\overline}

\def\n{\nabla}
\def\<{\langle}
\def\>{\rangle}

\def\n{\nabla}

\def\RR{\mathbb{R}}

\def\SS{\mathbb{S}}

\def\a{\alpha}

\def\l{\lambda}

\def\s{\sigma}

\def\ov{\overline}

\def\R{\mathbb{R}}

\def\wt{\widetilde}

{\left\{\begin{array}{@{}l@{}}}{\end{array}\right.}
\patchcmd{\abstract}{\scshape\abstractname}{\textbf{\abstractname}}{}{}
\makeatletter 
\def\@makefnmark{} 
\makeatother

\numberwithin{equation}{section}
\numberwithin{exa}{section}

\makeatletter
\@namedef{subjclassname@2020}{\textup{2020} Mathematics Subject Classification}
\makeatother

\begin{document}
\title [Prescribed $L_p$ quotient curvature problem]{Prescribed $L_p$ quotient curvature problem and related eigenvalue problem}
\author[X. Mei]{Xinqun Mei}
\address[X. Mei]{School of Mathematical Sciences, University of Science and Technology of China, Hefei, 230026, P.R.China. Mathematisches Institut, Albert-Ludwigs-Universit\"{a}t Freiburg, Freiburg im Breisgau, 79104, Germany}

\email{\href{mailto:xinqun.mei@math.uni-freiburg.de} {xinqun.mei@math.uni-freiburg.de}}
	
\author[G. Wang]{Guofang Wang}
\address[G. Wang]{Mathematisches Institut, Albert-Ludwigs-Universit\"{a}t Freiburg, 
		Freiburg im Breisgau, 79104, Germany}
\email{\href{mailto:guofang.wang@math.uni-freiburg.de} {guofang.wang@math.uni-freiburg.de}}
 
\author[L. Weng]{Liangjun Weng}
\address[L. Weng]{Dipartimento di Matematica, Universit\`a degli Studi di Roma "Tor Vergata",   Roma, 00133, Italy. Dipartimento di Matematica, Università di Pisa,  Pisa, 56127, Italy}

 \email{\href{mailto:liangjun.weng@uniroma2.it}  {liangjun.weng@uniroma2.it}}



\begin{abstract}
In this paper, we investigate the existence of admissible (and strictly convex) smooth solutions to the prescribed $L_p$ quotient type curvature problem with $p>1$. For cases where $p=k-l+1$ and $p> k-l+1$, we obtain an admissible solution without any additional conditions, which is strictly spherically convex under a convexity condition. Under the same convexity condition, we establish the existence of a strictly spherically convex solution for the case $p<k-l+1$, provided that the prescribed function is even, a condition known to be necessary.
\end{abstract}
\maketitle


\section{Introduction}
In this paper we are interested in  the following prescribed $L_{p}$ quotient type curvature problem in $\RR^{n+1}$, namely

\ 

\textit{Given any positive smooth function $f$ on $\SS^{n}$,  can one find a closed, strictly convex hypersurface $\S\subset\RR^{n+1}$ satisfying
 \begin{eqnarray}\label{quo-equ}
\frac{H_{k}(\kappa(X))}{H_{l}(\kappa(X))}=f(\nu(X))\<X, \nu(X)\>^{1-p}   ~~?
\end{eqnarray}}

\ 

\noindent Here $p>1$, $0\leq l<k\leq n$, $\kappa(X)=(\kappa_{1},\cdots, \kappa_{n})$ and $\nu(X)$ are the principal curvatures and the outward unit normal of $\S$ at $X$ respectively, and  $H_{k}(\kappa(X))$ is the normalized $k$th mean curvature (cf. Section \ref{sec2}).  It is equivalent  to finding a strictly spherical convex solution (or admissible solution) $u$ on $\SS^n$ satisfying 
\begin{eqnarray}\label{support-equ}
   \frac{H_{n-l}(A)}{H_{n-k}(A)}=\varphi u^{p-1},
\end{eqnarray}
where $\varphi:=f^{-1}$ and the matrix $A:=\n^{2}u+u\sigma$, $\n^{2}u$ is the Hessian of $u$ w.r.t the standard round metric $\sigma$ on $\SS^{n}$. We call the function $u\in C^{2}(\SS^{n})$ is strictly spherical convex if $A>0$ as matrix. For the admissible solution see Definition \ref{admissble sol} below. Equation  \eqref{quo-equ}  is a self-similar solution of the following anisotropic curvature ﬂow  
\begin{eqnarray*}
\frac d {dt} X = \left (f(\nu)\frac{H_l(\kappa(X))}{H_k(\kappa(X))} \right)^{\frac 1 {1-p}} \nu
\end{eqnarray*}
and serves as singularity models of this flow.
 
For $p=1$, Equation  \eqref{quo-equ} corresponds to the prescribed Weingarten curvature problem in classical differential geometry. 
In particular, if further $k=n$ and $l=0$, it reduces to the classical Minkowski problem, which has been completely  solved,  through the work of Minkowski
\cite{Min}, Alexandrov \cite{Alex}, Lewy \cite{Lewy}, Nirenberg \cite{Nire}, Pogorelov \cite{Pog52}, Cheng-Yau \cite{CY76} and many other mathematicians.    
For the case $0=l<k<n$,  Guan-Guan \cite{GG02}  proved the existence of an origin-symmetric, strictly convex hypersurface that satisfies Equation  \eqref{quo-equ} when $f$ is even (or group invariant). Subsequently, 
 Sheng-Trudinger-Wang \cite{STW} dropped the evenness assumption at the cost of introducing an exponential weight factor in the equation. Under the assumptions of evenness and an additional convexity on $f$, Guan-Ma-Zhou \cite{GMZ}  solved Equation  \eqref{quo-equ} for $1\leq l<k<n$. It is worth noting that, except for the Minkowski problem,  finding the necessary conditions, and certainly sufficient conditions
 for the prescribed Weingarten curvature problem remains a challenging problem.  For more related research, see for instance \cite{BIS2023, CNS-4, Chern,  CW00,   Ger, Ger-2, GLM06, GM, TW} and references therein.

 When $p\neq 1$, in the case of $k=n$ and $l=0$, Equation  \eqref{support-equ} transforms into the celebrated $L_{p}$ Minkowski problem initiated by Lutwak \cite{Lut93}. This encompasses the logarithmic Minkowski problem ($p=0$, refer to \cite{BLYZ}) and the centro-affine Minkowski problem ($p=-1-n$, refer to \cite{CW06}).  If $k=n$ and $1\leq l<n$, Equation  \eqref{support-equ} becomes the $L_{p}$ Christoffel-Minkowski problem. Hu-Ma-Shen \cite{HMS2004} and Guan-Xia \cite{GX} solved the cases of $p\geq n-l+1$ and $1<p<n-l+1$ respectively. For more significant progress related to the $L_{p}$ Minkowski problem and further the $L_{p}$ Christoffel-Minkowski problem, see for instance \cite{CX, GL, GLW, LW, Lut96, Lut04}  and reference therein.

Inspired by the aforementioned $L_{p}$ Minkowski-type problem, Lee \cite[Theorem 1.2 for $p=k+1$ and Theorem 1.4 for $p>k+1$]{Lee} and Hu-Ivaki \cite[Theorem 1.1 for $1<p<k+1$]{HI} recently achieved the following result concerning the prescribed $L_{p}$ Weingarten curvature problem (where $1\leq k< n$ and $l=0$ for Equation  \eqref{quo-equ}). 

\

\noindent{\bf Theorem A}. 
Let $l=0$ and $1\leq k <n$. For any positive smooth function $f$ on $\SS^n$.
    \begin{enumerate}
       
         \item If $p=k+1$, and $f$ is even, 
         then there exists a  unique (up to a dilation) 
         origin-symmetric, smooth strictly convex hypersurface $\S$ and  a positive constant $\gamma$ satisfying  
        \begin{eqnarray*} 
           \<X, \nu(X)\>^{k} H_{k}(\kappa(X))=\gamma f(\nu(X)),\quad \forall ~ X\in\S.
        \end{eqnarray*}
 \item If $p>k+1$,  then there exists a unique, smooth strictly convex hypersurface $\S$ satisfying \eqref{quo-equ}.
        
       \item If $1<p<k+1$, $f$ is even,  
       then there exists an origin-symmetric, smooth strictly convex hypersurface satisfying \eqref{quo-equ}. 
    \end{enumerate}

The primary objective of this paper is to extend Theorem A to the general case $l \geq 0$ of Equation  \eqref{quo-equ}. Firstly, it is crucial to recognize a significant distinction between the cases $l=0$ and $l>0$. In the case of $l=0$, if a solution exists, it automatically exhibits convexity. However, in the case of $l>0$, we need to consider the admissible solutions, which are not necessarily strictly spherical convex. Thanks to the seminal work of Bian-Guan \cite{BG}, we can establish convexity, provided that $f$ satisfies an additional condition, i.e., \eqref{f-convex}.

For $1<p<k+1$, it is well-established that the evenness assumption in Theorem A (3) is indispensable; without it, the existence of a smooth solution cannot be guaranteed (see for instance Guan-Xia \cite[Section 5 and Proposition 5.1]{GX} or Guan-Lin \cite[Theorem 3 and Section 4]{GL}). However, unlike the case $1<p<k+1$, where the evenness assumption is necessary, for the case $p\geq k+1$, it often appears redundant, as demonstrated in Guan-Lin \cite[Theorem 2]{GL} and Hu-Ma-Shen \cite[Theorem 1]{HMS2004}, among others. Therefore, it is natural to ask whether the assertion in Theorem A (1) is still true without the evenness condition required in Theorem A (1). Our first result in this paper provides an affirmative answer to this question, not only for the case $l=0$ but also for the general case $l\geq 0$.

\begin{thm}\label{thm_new1}
Let $p=k-l+1$ with $0\leq l< k <n$. For any positive  smooth function $f$  on $\SS^{n}$, then there exists a  unique  pair $(u,\tau)$, a positive, smooth and $(n-l)$-admissible function $u$ on $\SS^{n}$  and a positive constant $\tau$ satisfying the following  eigenvalue problem
        \begin{eqnarray}\label{critical-u}
            \frac{H_{n-l}(A)}{H_{n-k}(A)}=\tau  u^{k-l} f^{-1}.
        \end{eqnarray} 
The uniqueness is up to a dilation.
\end{thm}

 When $l=0$, the case $p=k+1$ in \eqref{support-equ} or \eqref{critical-u} can be viewed as a principal eigenvalue problem for the prescribed curvature equations. The eigenvalue problem for fully nonlinear operators was studied first by Lions  \cite{Lions86}. See also \cite{Wang1994}, \cite{Geng1995} and \cite{LS}.

Next, we examine the case where $l\geq0$ and $p>k-l+1$. When $l=0$, Theorem A (2) was provided in \cite[Theorem 1.4]{Lee}. Remarkably, this result directly stems from a broader prescribed $k$-Weingarten curvature theorem established by Guan-Ren-Wang \cite[Theorem 1.5]{GWR}. However, in their work \cite[Theorem 1.2]{GWR}, they devised certain smooth positive functions, leading to situations where the general prescribed quotient-type curvature problem lacks a smooth solution. Consequently, we cannot directly apply their overarching result to the case $l>0$. Nonetheless, we can establish the existence result for our quotient Equation  \eqref{support-equ}.

\begin{thm}\label{thm_new2}
Let $p>k-l+1$ with $0\leq l< k <n$. For any positive smooth function $f$ on $\SS^{n}$, then there exists a unique, smooth and  $(n-l)$-admissible solution $u$ of Equation  \eqref{support-equ}.
        
\end{thm}

Furthermore, we provide a sufficient condition to ensure that the $(n-l)$-admissible solution $u$ of Equation  \eqref{support-equ} obtained in Theorem \ref{thm_new1} and in Theorem \ref{thm_new2} is strictly spherical convex, by using the method of Bian-Guan \cite{BG}.  
\begin{thm}\label{thm_new3} 
The $(n-l)$-admissible solutions $u$ of Equation  \eqref{support-equ} obtained in Theorem \ref{thm_new1} and Theorem \ref{thm_new2} 
are strictly spherical convex, provided that 
$f$ satisfies 
 \begin{eqnarray} \label{f-convex}
     \n^{2} \left(f^{\frac{1}{p+k-l-1}} \right)+f^{\frac{1}{p+k-l-1}}\sigma \geq 0.
 \end{eqnarray}  
    
\end{thm}

Finally, we investigate the case for $1<p<k+1-l$, with $1\leq l<k<n$.  As mentioned above, we have to assume the evenness condition in this case.   For this case, we have to restrict ourselves to the class of strictly spherically convex solutions.
When $k=n$, $  1\leq l<n$ and $1<p<n-l+1$, 
Guan-Xia \cite[Corollary~1.1]{GX} established the existence of a strictly spherically convex solution to Equation  \eqref{support-equ} under the assumption that $f$ is even and spherical convex. The key lies in establishing uniformly $C^{0}$ estimate for origin-symmetric and strictly spherical convex solutions, especially a positive lower bound. Their argument relies on a weighted gradient estimate, the Alexandrov-Fenchel inequalities, and an ODE argument, but it appears not to be applicable to the Hessian-quotient type Equation  \eqref{support-equ}.   Recently, Hu-Ivaki \cite{HI} solved the case $l=0, 1\leq k< n$ for \eqref{quo-equ}.  The key ingredient in their argument is to establish the relationship between the sum of the principal radius and the radio of $R^{2}$ to $r$, where $R$ and $r$ denote the outer and inner radii of the convex hypersurface $\S\subset \RR^{n+1}$.  
Using the weighted gradient estimates and  sophisticated analysis, they first established the following key inequality
\begin{eqnarray}\label{hu-1}
    \sigma_{1}(A)\leq C\frac{R^{2}}{r}\left(\frac{r}{R}\right)^{\gamma},
\end{eqnarray}
where $\gamma\in \left(0, \frac{2(p-1)}{k}\right)$.  
Together with a geometric Lemma by Chou-Wang \cite[Lemma~2.2]{CW00}, which states
\begin{eqnarray}\label{hu-2}
    \frac{R^{2}}{r}\leq C(n)\sigma_{1}(A),
\end{eqnarray}
then Hu-Ivaki \cite{HI} obtained the following non-collapsing estimate 
\begin{eqnarray*}\label{hu-3}
   \frac{R}{r}\leq C. 
\end{eqnarray*}
Due to  $1<p<k+1$, the maximum principle implies a positive upper bound of $r$. Then they obtained the positive lower and upper bounds on the support function and the $C^{2}$ estimates simultaneously.

Thanks to a recent new result of Guan in \cite{G}, we will provide a simpler proof of the result of Hu-Ivaki \cite{HI}, which also works for the general case $l>0$ of Equation  \eqref{support-equ}.

\begin{thm}\label{thm-new4}
    Let $1<p<k+1-l$ with $1\leq l<k<n$. For any positive even smooth function $f$ 
    satisfying \eqref{f-convex}, then there exists a even, strictly spherically convex solution $u$ solving \eqref{support-equ}. Moreover, there exists a positive constant $C$, such that 
    \begin{eqnarray*}
        \|u\|_{C^2(\SS^n)}\leq C,
    \end{eqnarray*}
    where   $C$ depends on $n, k,l, p, \min\limits_{\SS^{n}}f$ and $\|f\|_{C^{2}(\SS^n)}$.
\end{thm}

 As a direct consequence of Theorem \ref{thm_new3} and Theorem \ref{thm-new4}, we have the following existence result for the prescribed $L_p$ quotient curvature problem \eqref{quo-equ}.
\begin{thm}\label{thm-quo-curv-prob}
    Let  $0\leq l< k <n$
    and let $f$ be a positive smooth function on $\SS^n$ satisfying 
    \eqref{f-convex}.
    \begin{enumerate}
        \item If $p=k-l+1$, 
         then there exists a  unique, smooth strictly convex hypersurface $\S$ and  a positive constant $\gamma$ satisfying  
        \begin{eqnarray*}
           \<X, \nu(X)\>^{k-l} H_{k}(\kappa(X))=\gamma f(\nu(X)),\quad \forall ~ X\in\S.
        \end{eqnarray*}
 \item If $p>k-l+1$,  then there exists a unique, smooth strictly convex hypersurface $\S$ satisfying \eqref{quo-equ}.
        
       \item If $1<p<k-l+1$, $f$ is even, 
       then there exists an origin-symmetric, smooth strictly convex hypersurface satisfying \eqref{quo-equ}. 
    \end{enumerate}
\end{thm}

In the rest of this paper, we assume that $f:\SS^n\to \R$ is always a positive smooth function. We will indicate it when we need the evenness assumption of $f$.

\

\textbf{The rest of the article is structured as follows.} 
In Section \ref{sec2}, we gather some properties of elementary symmetric functions. In Section \ref{sec-3}, we derive a priori estimates for admissible solutions to Equation  \eqref{support-equ} and complete the proof of Theorem \ref{thm_new1} and Theorem \ref{thm_new2}, respectively. Next, we establish a convexity criterion for the solution of Equation  \eqref{support-equ} and complete the proof of Theorem \ref{thm_new3}. In the last section, Section \ref{sec-4}, we establish a priori estimates of the even, strictly spherical convex solutions to Equation  \eqref{support-equ} when $1<p<k-l+1$ and complete the proof of Theorem \ref{thm-new4} using the topological degree theory.

\section{Preliminaries}\label{sec2}
In this section, we collect some well-known properties of elementary symmetric functions.

\begin{defn}
    Let $A=\{a_{ij}\}$ be an $n\times n$ symmetric matrix,    
    \begin{eqnarray}\label{k-ele}
        \sigma_{k}(A)=\sigma_{k}(\lambda(A))=\sum\limits_{1\leq i_{1}<i_{2}\cdots< i_{k}\leq n}\lambda_{i_{1}}\lambda_{i_{2}}\cdots \lambda_{i_{k}},
    \end{eqnarray}
    where $\l:=\lambda(A)=(\lambda_{1}(A), \lambda_{2}(A), \cdots, \lambda_{n}(A))$ is the  set of eigenvalues of  $A$. 

\end{defn}

We use the convention that $\sigma_0=1$ and $\sigma_k =0$ for $k>n$. Let $H_k(\lambda)$ be the normalization of 
$\sigma_{k}(\lambda)$ given by \begin{eqnarray*}
    H_k(\lambda):=\frac{1}{\binom{n}{k}}\s_k(\lambda).
\end{eqnarray*} Denote  $\sigma _k (\lambda \left| i \right.)$ the symmetric
	function with $\lambda_i = 0$ and $\sigma _k (\lambda \left| ij \right.)$ the symmetric function with $\lambda_i =\lambda_j = 0$.  Recall that the  G{\aa}rding cone is defined as
\begin{eqnarray}\label{2.4}
\Gamma_k := \{ \lambda  \in \mathbb{R}^n :\sigma _i (\lambda ) > 0,~~\forall 1 \le i \le k\}.
\end{eqnarray} 

\begin{defn}\label{admissble sol}
     A function $u\in C^2(\SS^n)$ of Equation  \eqref{support-equ} is called $k$-admissible solution if $$A=\n^2 u(x)+u(x)\s\in \Gamma_{k}$$ for all point $x\in\SS^n.$
     In particular, if $A\in \G_n$, we call $u$ strictly spherically convex.
\end{defn}
\begin{lem}\label{prop2.1}Let $\lambda=(\lambda_1,\cdots,\lambda_n)\in\mathbb{R}^n$ and $k
=1, \ldots, n$. Then
	\begin{enumerate}
\item $ \sigma_k(\lambda)=\sigma_k(\lambda|i)+\lambda_i\sigma_{k-1}(\lambda|i),  \quad  \forall 1\leq i \leq n.$
\item $ \sum\limits_{i = 1}^n {\sigma_{k}(\lambda|i)}=(n-k)\sigma_k(\lambda).$
\item $\sum\limits_{i = 1}^n {\lambda_i \sigma_{k-1}(\lambda|i)}=k\sigma_k(\lambda)$.
\item $\sum\limits_{i = 1}^n {\lambda_i^2 \sigma_{k-1}(\lambda|i)}=\s_1(\lambda)\sigma_k(\lambda)-(k+1)\s_{k+1}(\lambda)$.
	\end{enumerate}
\end{lem}
We denote $\sigma _k(A\left|
i \right.)$ the symmetric function with $A$ deleting the $i$-row and
$i$-column and $\sigma _k (A \left| ij \right.)$ the symmetric
function with $A$ deleting the $i,j$-rows and $i,j$-columns. 
\begin{lem}\label{prop2.2}
Suppose that  $A=\{A_{ij}\}$ is diagonal, and $k$ is a positive integer,
then
\begin{eqnarray*}
\sigma_{k-1}^{ij}(A)= \begin{cases}
\sigma _{k- 1} (A\left| i \right.), &\text{if } i = j, \\
0, &\text{if } i \ne j,
\end{cases}
\end{eqnarray*}
where $\sigma_{k-1}^{ij}(A)=\frac{{\partial \sigma _k (A)}} {{\partial A_{ij} }}$.
\end{lem}
\begin{lem}\label{pro-2.3}
    The following three properties hold.
    \begin{enumerate}
        \item  For $\lambda \in \G_k$ and $k > l \geq 0$, $ r > s \geq 0$, $k \geq r$, $l \geq s$, there holds 
\begin{eqnarray*} \label{1.2.6}
\left(\frac{H_{k}(\lambda)}{H_{l}(\lambda)}\right)^{\frac{1}{k-l}}\leq \left(\frac{H_{r}(\lambda)}{H_{s}(\lambda)}\right)^{\frac{1}{r-s}},
\end{eqnarray*}
with equality holds if and only if $\lambda_1 = \lambda_2 = \cdots =\lambda_n >0$.

\item For $0\le l<k\le n$, $\left(\frac{H_k(\lambda)}{H_{l}(\lambda)}\right)^{\frac{1}{k-l}}$ is a concave function for $\l\in\Gamma_
k$.
\item For $0\leq l<k\leq n$, denote  $F(\lambda):=\left(\frac{H_{k}(\lambda)}{H_{l}(\lambda)}\right)^{\frac{1}{k-l}}$. 
If $\lambda_{1}\geq \lambda_{2}\geq\cdots\geq \lambda_{n}$, there holds 
    \begin{eqnarray*}
      \frac{\partial F(\lambda)}{\partial \lambda_{1}}\leq \frac{\partial F(\lambda)}{\partial\lambda_{2}}\leq \cdots\leq \frac{\partial F(\lambda)}{\partial\lambda_{n}}.  
    \end{eqnarray*}
    \end{enumerate}
\end{lem}
For a proof of Lemma \ref{prop2.1} to Lemma \ref{pro-2.3},  one can refer to \cite[Lemma~2.10, Theorem~2.11 and Lemma~1.5]{Spruck} respectively.

\section{Case $p\geq k-l+1$}\label{sec-3}
In this section, we consider the case $p\geq k-l+1$. Firstly, we give the proof of Theorem \ref{thm_new1} and Theorem \ref{thm_new2} for the admissible solution. To deal with the prescribed $L_p$ quotient curvature problem \eqref{quo-equ}, we need to establish the existence of convex solutions. Secondly, we provide a sufficient condition to ensure that the admissible solution obtained in Theorem \ref{thm_new1} and Theorem \ref{thm_new2} is spherical convex.

\subsection{A priori estimates} \ 

We establish a priori estimates for the positive, admissible solution to Equation  \eqref{support-equ}.  A key ingredient is that we can obtain uniform estimates for the modified solution to Equation  \eqref{support-equ}. 
While $p>k-l+1$ ensures the uniform lower bound for the solutions of Equation  \eqref{support-equ}. We present the following theorem in this section as our initial result.
\begin{thm}\label{thm-priori}
    Let $p>k-l+1$ with $0\leq l<k< n$. Assume $u$ is a positive and $(n-l)$-admissible solution to Equation  \eqref{support-equ}. Then for each $\a\in (0,1)$,  there exist some constant $C$, depending on $n,k,l,p$, $\min\limits_{\SS^{n}}\varphi$, 
    and $\|\varphi\|_{C^{3}(\SS^{n})}$, such that
   \begin{eqnarray}\label{sch-1}
       \|u\|_{C^{4,\alpha}(\SS^{n})}\leq C.
   \end{eqnarray}
In particular, if $k-l+1< p\leq k-l+2$, the (rescaling) solution 
   \begin{eqnarray}
       \wt u:=\frac u {\min\limits_{\SS^n} u}
   \end{eqnarray}  satisfies
   \begin{eqnarray}\label{sch-2}
       \|\widetilde{u}\|_{C^{4,\alpha}(\SS^{n})}\leq C'
   \end{eqnarray}
   where the constant $C'$ depends only on $n,k,l$, $\min\limits_{\SS^{n}}\varphi$, 
   and $\|\varphi\|_{C^{3}(\SS^{n})}$, but not on $p$. 
\end{thm}

Before proving Theorem \ref{thm-priori}, certain preparations are necessary. In the following, for convenience,  when dealing with the tensors and their covariant derivatives on  $\SS^{n}$, we will use a local frame to express tensors with the help of their components and indices appearing after the semi-colon to denote the covariant derivatives. For example, let $\{e_i\}_{i=1}^n$ be an orthonormal frame on $\SS^{n}$,   the expression $u_{ij}$ denotes $\n^2 u(e_i,e_j)$ and   $A_{ijs}:=\n_{e_s}A({e_i, e_j})$ etc. We denote
\begin{eqnarray} \label{F(A)}  
F(A)=\left(\frac{H_{n-l}(A)}{H_{n-k}(A)}\right)^{\frac{1}{k-l}},~~  F^{ij}=\frac{\partial F(A)}{\partial A_{ij}},~~ F^{ij,pq}=\frac{\partial^{2}F(A)}{\partial A_{ij}\partial A_{pq}},
\end{eqnarray}
and 
\begin{eqnarray}\label{F(A) bar}
\overline{F}(A)=\frac{H_{n-l}(A)}{H_{n-k}(A)},~~  \ov{F}^{ij}=\frac{\partial \ov F(A)}{\partial A_{ij}},~~ \bar{F}^{ij,pq}=\frac{\partial^{2} \ov F(A)}{\partial A_{ij}\partial A_{pq}}.
\end{eqnarray}
We use the Einstein summation convention, wherein the indices that appear multiple times are automatically summed over, regardless of whether they are in the upper or lower position, if there is no confusion.

First, we derive the $C^{0}$ estimate for the positive admissible solution to Equation  \eqref{support-equ} for $p>k-l+1$, which is easy and follows from the maximum principle.
\begin{lem}\label{lem-c0}
   Let $p> k-l+1$ with $0\leq l<k<n$. Suppose $u$ is a positive, $(n-l)$-admissible solution to Equation  \eqref{support-equ}. Then 
    \begin{eqnarray}\label{C0-est}
     \min\limits_{ \SS^{n}}\varphi^{-1} \leq  u^{p-(k-l+1)}\leq \max\limits_{ \SS^{n}}\varphi^{-1} .
    \end{eqnarray}
\end{lem}
    \begin{proof}
        Assume $u$ attains its minimum value at some point $x_{0}\in\mathbb{S}^{n}$.  Then, at $x_{0}$, there holds
        \begin{eqnarray*}
            \n u=0,  ~~\text{and}~~\n^{2}u\geq 0,
        \end{eqnarray*}
       taking into account of Equation  \eqref{support-equ},
            \begin{eqnarray*}
                u^{k-l}(x_{0})\leq \frac{H_{n-l}(A)}{H_{n-k}(A)}(x_{0})=\varphi(x_{0})u^{p-1}(x_{0}),
            \end{eqnarray*}
     which implies 
     \begin{eqnarray*}
         u^{p-(k-l+1)}(x_{0})\geq \min\limits_{x\in\SS^{n}}\varphi^{-1}(x).
     \end{eqnarray*}
Similarly, there holds  
     \begin{eqnarray*}
         u^{p-(k-l+1)}\leq \max\limits_{x\in\SS^{n}}\varphi^{-1}(x).
     \end{eqnarray*}Hence we complete the proof.
 \end{proof}

Next, we establish the logarithmic gradient estimate for the admissible solution $u$ of Equation  \eqref{support-equ}. This estimate will play a crucial role in investigating the eigenvalue problem, i.e., the case $p=k-l+1$ of Equation  \eqref{support-equ}.   In particular, when $l=0, k=n$ and $l>0,k=n$, similar results were proved in \cite[Section 2]{GL} and \cite[Section 3]{HMS2004} respectively. 
\begin{lem}\label{gradient}
      Let $p\geq k+1-l$ with $0\leq l<k<n$. Suppose  $u$ is a positive, $(n-l)$-admissible solution to Equation  \eqref{support-equ}. Then there exists a constant $C$ depending on $n, k,l$, $\min\limits_{\SS^{n}}\varphi, 
      \| \varphi\|_{C^{1}(\SS^{n})}$ and the upper bound of $p_{0}:=p-(k-l+1)$, such that
      \begin{eqnarray*}
          \max\limits_{\SS^{n}}|\n \log u|\leq C.
      \end{eqnarray*}

\end{lem}
\begin{proof}
Define $v:=\log u$, from  Equation  \eqref{support-equ}, it is easy to see that   $v$ satisfies
\begin{eqnarray}\label{v-equ}
   G(B):=\frac{H_{n-l}(B)}{H_{n-k}(B)}=\varphi e^{p_{0}v}, \quad \text{on}~\SS^{n},
\end{eqnarray}
where $B:=(B_{ij})=(v_{ij}+v_{i}v_{j}+\sigma_{ij})$. 
Consider the  auxiliary function
\begin{eqnarray*}
    \psi:=\frac{1}{2}|\n v|^{2}.
\end{eqnarray*}
Assume $\psi$ attains its maximum value at some point $x_{0}\in\SS^{n}$. By choosing an orthonormal frame $\{e_{i}\}_{i=1}^{n}$ at $x_{0}$, such that 
\begin{eqnarray*}
\n v=v_1e_1 ~~ \text{and}\quad \n^{2}v~\text{is diagonal at } x_0.
\end{eqnarray*}
Denote $G^{ij}=\frac{\partial G(B)}{\partial B_{ij}}$, then at $x_0$, it follows that $B_{ij}$ and $G^{ij}$ are also diagonal. 
In the following, we compute at $x_0$. By the critical and maximal condition, 
\begin{eqnarray}\label{one deri}
    0=\psi_{i}=v_{s}v_{si},
\end{eqnarray}
and 
\begin{eqnarray}\label{sec deri}
\begin{aligned}
 0\geq G^{ij} \psi_{ij}
    =G^{ij}(v_{si}v_{sj}+v_{s}v_{sij}).
    \end{aligned}
\end{eqnarray}

From \eqref{one deri}, we have
\begin{eqnarray}\label{v11}
v_{11}=0.
\end{eqnarray}  By the Ricci identity on $\SS^{n}$,
\begin{eqnarray}\label{Ric-identity}
    v_{sij}=v_{ijs}+v_{s}\delta_{ij}-v_{j}\delta_{si},
\end{eqnarray}
it follows
\begin{eqnarray}\label{diff}
    G^{ij}v_{sij} 
    =G^{ij}v_{ijs}+v_{s}\sum\limits_{i\neq s}G^{ii}.
\end{eqnarray}
Differentiating Equation  \eqref{v-equ} once, 
\begin{eqnarray}\label{one-dff}
    G^{ij}v_{ijs}=e^{p_{0}v}(p_{0}v_{s}\varphi+\varphi_{s})-2G^{ij}v_{i}v_{js}.
\end{eqnarray}
Substituting  \eqref{diff} and \eqref{one-dff} into  \eqref{sec deri}, in view of \eqref{v11}, we obtain
\begin{eqnarray}
    0&\geq &G^{ij}\psi_{ij}\notag\\
    &=&G^{ii}v_{ii}^{2}+v_{1}^{2}\sum\limits_{i=2}^{n}G^{ii}-2G^{11}v_{1}^{2}v_{11}+e^{p_{0}v}(p_{0}v_{1}^{2}\varphi+\varphi_{1}v_{1})\notag\\
    &\geq& v_{1}^{2}\sum\limits_{i=2}^{n}G^{ii}-e^{p_{0}v}|\varphi_{1}|v_{1}.\label{key-ine-a}
\end{eqnarray}
Now we \textbf{claim} that there exists some index $i_0\in \{2,\ldots, n\}$ such that 
\begin{eqnarray}\label{B11}
    B_{i_0i_0}\leq B_{11}.
\end{eqnarray}In fact, if the \textbf{claim}  is not true, it follows \begin{eqnarray*}
    e^{p_{0}v}\varphi(x_{0})=G(B)\geq G(B_{11}I)=B_{11}^{k-l}=(1+v_{1}^{2})^{k-l},
\end{eqnarray*}thus we complete the proof.

Under the \textbf{claim} \eqref{B11}, from Lemma \ref{pro-2.3} (3), we have 
\begin{eqnarray}\label{Gii}
    G^{i_{0}i_{0}}\geq G^{11}.
\end{eqnarray}
From Lemma \ref{pro-2.3} (1), we have
\begin{eqnarray*}
    \frac{ H_{n-l}(B)}{H_{n-l-1}(B)} \leq \left(\frac {H_{n-l}(B)}{H_{n-k}(B)}\right)^{\frac 1 {k-l}},
\end{eqnarray*}together with Lemma \ref{prop2.1} (2) and Lemma \ref{lem-c0}, we obtain
\begin{eqnarray}\label{sum Gii}
\sum\limits_{i=1}^{n}G^{ii}&=&\frac{C_{n}^{k}}{C_{n}^{l}}\sum\limits_{i=1}^{n}\frac{\sigma_{n-l-1}(B|i)\sigma_{n-k}(B)-\sigma_{n-l}(B)\sigma_{n-k-1}(B|i)}{\sigma_{n-k}^{2}(B)}\notag \\
&\geq& \frac{C_{n}^{k}}{C_{n}^{l}}\frac{(k-l)(l+1)}{n-l}\frac{\sigma_{n-l-1}(B)}{\sigma_{n-k}(B)}\notag\\
&\geq& C(n, k, l)(\varphi e^{p_{0}v})^{\frac{k-l-1}{k-l}}\geq c_{0}>0,
\end{eqnarray}
where the constant $c_{0}$ depends on $n,k,l,p_0$ and $\min\limits_{\SS^{n}}\varphi$  and $\max\limits_{\SS^{n}}\varphi$. 
Substituting \eqref{Gii} and  \eqref{sum Gii} into \eqref{key-ine-a}, we derive \begin{eqnarray*}
    |v_{1}|\leq C.
\end{eqnarray*}
In conclusion, we complete the proof.
\end{proof}
As a direct consequence of Lemma \ref{gradient}, we have the following estimates for $$\widetilde{u}:=\frac{u}{\min\limits_{\SS^{n}}u}.$$
\begin{cor}\label{coro}
    Let $p> k+1-l$ and $u$ be a positive, $(n-l)$-admissible solution to Equation  \eqref{support-equ}. Then there exists a constant $C$ depending on $n, k,l$, $\min\limits_{\SS^{n}}\varphi$, $\|\varphi\|_{C^{1}(\SS^{n})}$ and the upper bound of $p_{0}$,  such that
    \begin{eqnarray*}
        1\leq \widetilde{u}\leq C,
    \end{eqnarray*}
    and 
    \begin{eqnarray*}
        |\n \widetilde{u}|\leq C.
    \end{eqnarray*}
\end{cor}

At the end of this section, we derive the $C^{2}$ estimates for the positive, $(n-l)$-admissible solution to Equation  \eqref{support-equ} when $p>k-l+1$. To facilitate the study of the corresponding eigenvalue problem, we  derive the $C^{2}$ estimates for the function $\widetilde{u}$ which satisfies the following equation
\begin{eqnarray}\label{wu equ}
    \frac{H_{n-l}( A)}{H_{n-k}(  A)}=\widetilde{\varphi}  u^{p-1},\quad \text{on}~\SS^{n},
\end{eqnarray}
for some  positive function  $\wt\varphi$ defining on $\SS^n$. By Lemma \ref{lem-c0}, we know that $\wt \varphi$ satisfies 
\begin{eqnarray}\label{f condi}
     \frac{\min\limits_{\SS^{n}}\varphi}{\max\limits_{\SS^{n}}\varphi}\leq \widetilde{\varphi}(x)\leq \frac{\max\limits_{\SS^{n}}\varphi}{\min\limits_{\SS^{n}}\varphi}\quad ~\forall x \in\SS^{n},\quad\text{and}~ \|\widetilde{\varphi}\|_{C^{2}(\SS^{n})}\leq C\|\varphi\|_{C^{2}(\SS^{n})}.
\end{eqnarray}

\begin{lem}\label{lem-C2}
Let $p\geq 1$  and $0\leq l< k<n$. Suppose $u$ is a positive, $(n-l)$-admissible solution to Equation  \eqref{wu equ}. If  there exists a positive constant $C_{0}$ which  depends on  $n, k,l$, $\min\limits_{\SS^{n}}\varphi$,
$\| \varphi\|_{C^{1}(\SS^{n})}$ and the upper bound of $p_{0}$, such that $\|u\|_{C^{1}(\SS^{n})}\leq C_{0}$, then  it holds
\begin{eqnarray*}
\Delta u+nu\leq C.
\end{eqnarray*}
The constant $C$  depends on $n,k,l, C_{0}$ and $\|\varphi\|_{C^{2}(\SS^{n})}$.
\end{lem}
\begin{proof}
We define the function 
    \begin{eqnarray*}
        \Phi=\Delta u+nu.
    \end{eqnarray*}
 Suppose $\Phi$ attains its maximum value at some point $x_{0}\in \SS^{n}$.  By choosing a local orthonormal frame $\{e_{i}\}_{i=1}^{n}$ around $x_{0}$ such that $A=\n^{2}u+u\sigma$ is diagonal, it follows $F^{ij}$ is also diagonal at $x_{0}$. Below we conduct the computation at $x_{0}$. By  maximal condition, we have
 \begin{eqnarray}
     0\geq F^{ij}\Phi_{ij}=\sum\limits_{s=1}^{n} F^{ij} A_{ssij}. 
     \label{two-der}
 \end{eqnarray}
Using the following commutator formula on $\SS^{n}$
 \begin{eqnarray}\label{4-change}
     u_{srij}=u_{ijrs}+2u_{rs}\delta_{ij}-2u_{ij}\delta_{rs}+u_{si}\delta_{rj}-u_{rj}\delta_{is},
 \end{eqnarray}
we have
\begin{eqnarray}\label{Asi}
\sum\limits_{s=1}^{n}A_{ssii}=\sum\limits_{s=1}^{n}A_{iiss}-nA_{ii}+\Phi.
\end{eqnarray}
From \eqref{F(A)}, we rewrite Equation  \eqref{wu equ} as
 \begin{eqnarray}\label{homo-one-equ}
    F(A)=\phi,
 \end{eqnarray}
 where  $\phi:=(\widetilde{\varphi}u^{p-1})^{\frac{1}{k-l}}$. 
Differentiating Equation  \eqref{homo-one-equ} twice, 
\begin{eqnarray}\label{ek-two} F^{ij,pq}A_{ijs}A_{pqs}+F^{ij}A_{ijss}=\phi_{ss}.
\end{eqnarray}
Combining \eqref{two-der}, \eqref{Asi} and \eqref{ek-two}, we derive that
\begin{eqnarray}
    0&\geq& F^{ij}\Phi_{ij}=\Phi\sum\limits_{i=1}^{n}F^{ii}+(\Delta\phi -n\phi)-F^{ij,pq}A_{ijs}A_{pqs}\notag\\
    &\geq &\Phi\sum\limits_{i=1}^{n}F^{ii}+(\Delta\phi -n\phi),\label{key-ine}
\end{eqnarray}
where the last inequality used the concavity of $F(A)$ (cf. Lemma \ref{pro-2.3} (2)). Combining $\|u\|_{C^{1}(\SS^{n})}\leq C_{0}$ and \eqref{f condi}, we have
\begin{eqnarray}\label{phi}
    0\leq \phi\leq C',
\end{eqnarray}
where
the constant $C'$ depends on  $k, l, C_{0}, \min\limits_{\SS^{n}}\varphi$, and $ \max\limits_{\SS^{n}}\varphi$. Direct calculations yield
\begin{eqnarray*}
  \Delta \phi&=&u^{\frac{p-1}{k-l}}\Delta(\widetilde{\varphi}^{\frac{1}{k-l}})+ \frac{2(p-1)}{(k-l)}u^{\frac{p-(k+1-l)}{k-l}} \<\n(\widetilde{\varphi}^{\frac{1}{k-l}}),\n u\>\\
  &&+\frac{(p-1)(p-(k+1-l))}{(k-l)^{2}}\widetilde{\varphi}^{\frac{1}{k-l}}u^{\frac{p-(2k+1-2l)}{k-l}}|\n u|^{2}\\
  &&+\frac{p-1}{k-l}\widetilde{\varphi}^{\frac{1}{k-l}} u^{\frac{p-(k+1-l)}{k-l}}(\Phi-n u).
\end{eqnarray*}
We assume $\Phi$ is sufficiently large enough, otherwise we have done. From $p\geq 1$, then we have 
\begin{eqnarray}\label{3}
    \Delta\phi \geq -C.
\end{eqnarray}
Similarly as \eqref{sum Gii}, we get
\begin{eqnarray}\label{sum Fii-1}
    \sum\limits_{i=1}^{n}F^{ii}&=&\frac{1}{k-l}\left[\frac{H_{n-l}(A)}{H_{n-k}(A)}\right]^{\frac{l+1-k}{k-l}}\sum\limits_{i=1}^{n}\bar {F}^{ii}
    \geq C(n,k, l, \|\varphi\|_{C^0})>0.
\end{eqnarray}
Substituting \eqref{phi}, \eqref{3}, and \eqref{sum Fii-1} into \eqref{key-ine}, we  conclude that $$\Phi\leq C,$$ and 
complete the proof of Lemma \ref{lem-C2}.
\end{proof}

Now we are ready to prove Theorem \ref{thm-priori}.
\begin{proof}[\textbf{Proof of Theorem \ref{thm-priori}}]\ 

In view of Lemma \ref{lem-c0}, Lemma \ref{gradient}, Corollary \ref{coro} and Lemma \ref{lem-C2}, we obtain
    \begin{eqnarray*}
        \|u\|_{C^{2}(\SS^{n})}\leq C, \quad \|\widetilde{u}\|_{C^{2}(\SS^{n})}\leq C.
    \end{eqnarray*}
Using the Evans-Krylov theorem and the Schauder theory, we get the Schauder estimate \eqref{sch-1} and \eqref{sch-2}.

\end{proof}

\subsection{Proof of Theorem \ref{thm_new2}}\ 

In this subsection, we utilize the method of continuity to establish the existence of $(n-l)$-admissible solutions to Equation   \eqref{support-equ} when $p> k-l+1$. Prior to this, certain preparations are required.

\begin{thm}\label{thm-uni}
    Let  $p>k-l+1$ and $0\leq l<k< n$. 
    Then the positive $(n-l)$-admissible solution to Equation  \eqref{support-equ} is unique. 
\end{thm}
\begin{proof}
    Suppose  there exists two positive, $(n-l)$-admissible solutions $u$ and $\widehat{u}$ to Equation  \eqref{support-equ} and  set $g=\log \frac{\widehat{u}}{u}$. Assume that $g$ attains its minimum at point $x_{0}\in\SS^{n}$. At $x_{0}$,
    \begin{eqnarray*}
        0=g_{i}=\frac{\widehat{u}_{i}}{\widehat{u}}-\frac{u_{i}}{u},
    \end{eqnarray*}
    and as matrices,
    \begin{eqnarray*}
        0&\leq& (g_{ij} )= \left(\frac{\widehat{u}_{ij}}{\widehat{u}}-\frac{u_{ij}}{u}\right),
    \end{eqnarray*}which implies
\begin{eqnarray*}
  \frac{1}{\widehat{u}}  \widehat{A}\geq \frac{1}{u}A,
\end{eqnarray*}
where $\widehat{A}=\n^{2}\widehat{u}+\widehat{u}\s$. Therefore
\begin{eqnarray*}
    \widehat{u}^{p-1}\varphi(x_{0})=\frac{H_{n-l}(\widehat{A})}{H_{n-k}(\widehat{A})}\geq \left(\frac{\widehat{u}}{u} \right)^{k-l}\frac{H_{n-l}(A)}{H_{n-k}(A)}=\widehat{u}^{k-l}u^{p-(k-l+1)}\varphi(x_{0}).
\end{eqnarray*}
Since $p>k-l+1$, we have $\widehat{u}\geq u $ at $x_{0}$ and $\min\limits_{\SS^{n}}g\geq 0$. Similarly, we can show $\max\limits_{\SS^{n}}g \leq 0$. Hence $\widehat {u}\equiv u$ on $\SS^{n}$.
\end{proof}

A direct consequence of Theorem \ref{thm-uni} yields the following result, by choosing $\varphi=1$ in Equation  \eqref{support-equ}. 
\begin{cor}
 Let  $p>k-l+1$ and $0\leq l<k< n$. Assume $u$ be a positive $(n-l)$-admissible solution of equation \begin{eqnarray}
        \frac{H_{n-l}(A)}{H_{n-k}(A)}=u^{p-1},
    \end{eqnarray} then $u\equiv 1$.
\end{cor}
\begin{lem}\label{lem-inver}
   Let  $p>k-l+1$ and $0\leq l<k< n$.  
   For any positive, $(n-l)$-admissible solution to Equation  \eqref{support-equ}, the corresponding linearization operator
   \begin{eqnarray*}
      \mathcal{L}_{u}h=F^{ij}(h_{ij}+h\delta_{ij})-\frac{p-1}{k-l}u^{\frac{p-1}{k-l}-1}\varphi^{\frac{1}{k-l}}h,\quad \text{}~h\in C^{2}(\SS^{n}),
   \end{eqnarray*}
   is invertible.
\end{lem}

\begin{proof}
    We follow the idea in Caffarelli-Nirenberg-Spruck \cite[Proof of Lemma 1.2]{CNS-4}  and show that  ${\rm{Ker}}(\mathcal{L}_{u})=\{0\}$. To proceed, we define the function  $$\ell:=\frac h u,$$ then
    \begin{eqnarray*}
        \mathcal{L}_{u}h=\left(1-\frac{p-1}{k-l}\right)u^{\frac{p-1}{k-l}}\ell \varphi^{\frac{1}{k-l}}+uF^{ij}\ell_{ij}+2F^{ij}\ell_{i}u_{j}.
    \end{eqnarray*}
If $h\in {\rm{Ker}}\mathcal{L}_{u}$, i.e., $\mathcal{L}_{u}h=0$, we have
\begin{eqnarray}\label{ker-1}
    \left(1-\frac{p-1}{k-l} \right)u^{\frac{p-1}{k-l}}\ell \varphi^{\frac{1}{k-l}}+uF^{ij}\ell_{ij}+2F^{ij}\ell_{i}u_{j}=0.
\end{eqnarray}
At the maximum value point of $\ell$, we have $\ell_{i}=0$ and $\ell_{ij}\leq 0$. From \eqref{ker-1}, we conclude 
\begin{eqnarray*}
   \left (1-\frac{p-1}{k-l} \right)u^{\frac{p-1}{k-l}}\ell \varphi^{\frac{1}{k-l}}\geq 0,
\end{eqnarray*}
which implies $$\max\limits_{\SS^{n}}\ell \leq 0.$$   Similarly, we can show that $$\min\limits_{\SS^{n}}\ell\geq 0.$$ Altogether we have $\ell\equiv 0$, in turn $h\equiv 0$. Hence we complete the proof. 
\end{proof}

\begin{proof}[\textbf{Proof of Theorem \ref{thm_new2}}] \ 

We employ the method of continuity to demonstrate the existence of the $(n-l)$-admissible solution to Equation  \eqref{support-equ} when $p>k-l+1$. Specifically, for $t\in[0, 1]$, we consider a one-parameter family of equations
\begin{eqnarray}\label{t-equ}
   \frac{H_{n-l}(A)}{H_{n-k}(A)}=u^{p-1}\varphi_{t},
\end{eqnarray}
where $$\varphi_{t}=\left((1-t)+t f^{\frac{1}{p+k-l-1}}\right)^{-(p+k-l-1)}.$$ Let 
\begin{eqnarray*}
    \mathcal{I}=\{t| 0\leq t\leq 1,\text{and}~\eqref{t-equ}~\text{has~a~positive}, (n-l)\text{-admissible~solution} ~u_t \}.
\end{eqnarray*}
Obviously, $\mathcal{I}$ is a non-empty set, for $u_{0}=1$ is a positive, $(n-l)$-admissible solution to Equation  \eqref{t-equ} with $\varphi_{0}=1$. The openness of $\mathcal{I}$ follows from Lemma \ref{lem-inver} and the implicit function theorem, while the closedness from  Theorem \ref{thm-priori}. 
Therefore we conclude that $\mathcal{I}=[0,1]$ and we get the existence of $(n-l)$-admissible solution to Equation  \eqref{support-equ}
when $p>k+1-l$. The uniqueness follows from  Theorem \ref{thm-uni}.
\end{proof}

\subsection{Proof of Theorem \ref{thm_new1}} \ 

Next, we adapt the technique in \cite[Section~3]{GL} and \cite[Section~4]{HMS2004}, by using an approximation argument to complete the proof of Theorem \ref{thm_new1}.
\begin{proof}[\textbf{Proof of Theorem \ref{thm_new1}}]\ 

For any $\varepsilon\in(0,1)$.  From Theorem \ref{thm_new2}, there exists a unique, smooth, and $(n-l)$-admissible function $u_{\varepsilon}$  satisfying 
\begin{eqnarray*}
    \frac{H_{n-l}(A)}{H_{n-k}(A)}=\varphi u^{k-l+\varepsilon}.
\end{eqnarray*}
Set $\widetilde{u}_{\varepsilon}=\frac{u_{\varepsilon}}{\min\limits_{\SS^{n}}u_{\varepsilon}}$, then $\widetilde{u}_{\varepsilon}$ satisfies the following equation
\begin{eqnarray*}
    \frac{H_{n-l}(A)}{H_{n-k}(A)}=\varphi (\min\limits_{\SS^{n}}u_{\varepsilon})^{\varepsilon}u^{k-l+\varepsilon}. 
\end{eqnarray*}
By Theorem \ref{thm-priori},  we have
\begin{eqnarray*}
    \|\widetilde{u}_{\varepsilon}\|_{C^{4,\alpha}(\SS^{n})}\leq C^{'},
\end{eqnarray*}
where the constant $C^{'}$ depends on $n,k,l, \min\limits_{\SS^{n}}\varphi$, 
and $\|\varphi\|_{C^{3}(\SS^{n})}$, but it  is independent of $\varepsilon$. Hence, there exists a subsequence $\varepsilon_{j}\rightarrow 0$ as $j\rightarrow \infty$, such that $\widetilde{u}_{\varepsilon_{j}}\rightarrow \widetilde u$ in $C^{4}(\SS^{n})$. Lemma \ref{lem-c0} implies $(\min\limits_{\SS^{n}}u_{\varepsilon_{j}})^{\varepsilon_{j}}\rightarrow \tau$ for some constant $\tau$, which satisfies 
\begin{eqnarray*}
    \min\limits_{x\in\SS^{n}}\varphi^{-1}(x)\leq \tau\leq \max\limits_{x\in\SS^{n}} \varphi^{-1}(x).
\end{eqnarray*}
Hence $\widetilde u$ is a positive, $(n-l)$-admissible function satisfying 
\begin{eqnarray}\label{homo-gene}
    \frac{H_{n-l}(A)}{H_{n-k}(A)}=\tau \varphi u^{k-l},\quad \text{on}~\SS^{n}.
\end{eqnarray}

As for the uniqueness of positive constant $\tau$ and the $(n-l)$-admissible solution $u$ (up to to a dilation) to Equation  \eqref{critical-u}, one can follow 
 the same argument as outlined in \cite[Section~3]{GL} or \cite[Section~4]{HMS2004} by employing the maximum principle, for the concise of this paper, we just omit it here. Hence we complete the proof.
\end{proof}

\subsection{Proof of Theorem \ref{thm_new3}}\

Previously, for $p\geq k-l+1$ with $0\leq l <k<n$, we have established the existence of an admissible solution to Equation  \eqref{support-equ}. However, if $l\geq 1$, the $(n-l)$-admissible solution may not necessarily be strictly spherically convex. In this section, our objective is to establish a convexity criterion for the convex solutions to Equation  \eqref{support-equ} with $1\leq l<k<n$ and $p>1$. Consequently, we provide an affirmative answer to the $L_p$ quotient curvature type problem \eqref{quo-equ}.

The convexity criterion, expressed as \eqref{varphi convex}, resembles those found in previous works such as \cite{GM, GMZ, GX}, among others. The key component involves leveraging the constant rank theorem (cf. Bian-Guan \cite[Theorem~3.2]{BG}). 

\begin{thm}\label{thm convex}
    Let $p\geq 1$ and $1\leq l<k< n$. Suppose 
    $u$ is a positive, $(n-l)$-admissible solution to Equation  \eqref{support-equ} with $\n^{2}u+u\sigma\geq 0$ on $\SS^{n}$. If 
    $\varphi$ satisfies
   \begin{eqnarray}\label{varphi convex}
       \n ^{2} \left (\varphi^{-\frac{1}{p+k-l-1}} \right)+ \varphi^{-\frac{1}{p+k-l-1}}\sigma\geq 0,\quad \text{on}~~\SS^{n},
   \end{eqnarray}
   then 
   $A=\n^{2}u+u\sigma>0$ on $\SS^{n}$.
\end{thm}
In particular, Theorem \ref{thm convex} was established in Guan-Ma-Zhou \cite[Theorem 1.2]{GMZ} for $p=1$ and Guan-Xia \cite[Theorem 1.1]{GX} for $k=n$ and $1<p<n-l+1$, respectively.
\begin{proof}
   Suppose that $A=\n^{2}u+u\sigma$ attains its minimum rank $m$ at point $x_{0}\in\SS^{n}$ with $n-l\leq m\leq n-1$, then there exist a small neighborhood $\mathcal{O}\subset \SS^{n}$ of $x_{0}$ and a small constant $c_{0}$, such that 
    \begin{eqnarray*}
        \sigma_{m}(A)(x)>c_{0},~\text{in}~\mathcal{O}, \quad \text{and}~\sigma_{m+1}(A)(x_{0})=0.
    \end{eqnarray*}
We define the  function 
\begin{eqnarray*}
    \eta(x)=\sigma_{m+1}(A)+\frac{\sigma_{m+2}(A)}{\sigma_{m+1}(A)}.
\end{eqnarray*}
From Bian-Guan \cite[Corollary~2.2]{BG}, we know that $\eta\in C^{1.1}(\mathcal{O})$. 

For any fixed point $x\in\mathcal{O}$, choose a local orthonormal frame $\{e_{i}\}_{i=1}^{n}$ around $x$, such that at $x$, $A$ is diagonal and 
\begin{eqnarray*}
    A_{11}\leq A_{22}\leq \cdots\leq A_{nn}.
\end{eqnarray*}
Let $r[u](x)=(r_{1},\cdots, r_{n})$ be the eigenvalues of matrix $A$ at $x$, there exists a positive constant $C_{0}$ depending on $\|u\|_{C^{3,1}}$ and $r[u](x_{0})$ and $\mathcal{O}$, such that 
\begin{eqnarray*}
    r_{n}\geq r_{n-1}\geq \cdots\geq r_{n-m+1}\geq C_{0}.
\end{eqnarray*}
Let $\mathcal{G}=\{n-m+1, \cdots, n\}$ and $\mathcal{B}=\{1, \cdots, n-m\}$  be the ``good'' and ``bad'' set of indices respectively, as in \cite[Section 2]{BG}, and let  $\Lambda_{\mathcal{G}}=(r_{n-m+1},\cdots, r_{n})$  the ``good'' eigenvalues of $A$ at point $x$, and $\Lambda_{\mathcal{B}}=(r_{1},\cdots, r_{n-m})$ be the ``bad''  eigenvalues of $A$ at $x$. Without confusing, for the sake of brevity, we write $\mathcal{G}=\Lambda_{\mathcal{G}}$ and $\mathcal{B}=\Lambda_{\mathcal{B}}$. Therefore 
\begin{eqnarray*}
    \eta(x)\geq \sigma_{m+1}(A)\geq \sigma_{m}(\mathcal{G})\sigma_{1}(\mathcal{B}),
\end{eqnarray*}
then 
\begin{eqnarray}\label{bad index}
    A_{\alpha\alpha}=u_{\alpha\alpha}+u=O(\eta),\quad \alpha\in\mathcal{B}.
\end{eqnarray}

We define $\widetilde{F}(A)=-\left(\frac{\sigma_{k}(A)}{\sigma_{l}(A)}\right)^{-\frac{1}{k-l}}, \widetilde {F}^{ij}=\frac{\partial \widetilde{F}(A)}{\partial A_{ij}}$ and $\widetilde {F}^{ij, pq}=\frac{\partial^{2}\widetilde{F}(A)}{\partial A_{ij}\partial A_{pq}}$.
Following the same argument as in \cite[Thereom~3.2]{BG}, we obtain 
\begin{eqnarray}
    \sum\limits_{i,j=1}^{n}\widetilde{F}^{ij}\eta_{ij}&=& O(\eta +\sum\limits_{i,j\in \mathcal{B}}|\n A_{ij}|)-\sum\limits_{i=1}^{n}\widetilde{F}^{ii}\left(\frac{\sum\limits_{\alpha\in\mathcal{B}}V_{i\alpha}^{2}}{\sigma_{1}^{3}(\mathcal{B})}
    +\frac{\sum\limits_{\alpha\neq \beta\in\mathcal{B}}A_{\alpha\beta i}^{2}}{\sigma_{1}(\mathcal{B})}
    \right)\notag \\
    &&+\sum\limits_{\alpha\in\mathcal{B}}\left(\sigma_{m}(\mathcal{G})+\frac{\sigma_{1}^{2}(\mathcal{B}|\alpha)-\sigma_{2}(\mathcal{B}|\alpha)}{\sigma_{1}^{2}(\mathcal{B})}\right)
\left(\sum\limits_{i=1}^{n}\widetilde{F}^{ii}A_{\alpha\alpha ii}-2\sum\limits_{i,j\in\mathcal{G}}\widetilde{F}^{ii}\frac{A^{2}_{ji\alpha}}{A_{jj}}\right),~~~~~~\label{key-ine-2}
\end{eqnarray}
where $V_{i\alpha}=A_{\alpha\alpha i}\sigma_{1}(\mathcal{B})-A_{\alpha\alpha
}\sum\limits_{\gamma\in\mathcal{B}}A_{\gamma\gamma i}$.
 For $\alpha\in \mathcal{B}$, using \eqref{4-change}, we have
\begin{eqnarray}
\sum\limits_{i=1}^{n}\widetilde{F}^{ii}A_{\alpha\alpha ii}
&=&\sum\limits_{i=1}^{n}\widetilde{F}^{ii}\left(A_{ii\alpha\alpha}+A_{\alpha\alpha}-A_{ii}\right)\notag\\
&=&-\sum\limits_{i,j,p,q=1}^{n}\widetilde{F}^{ij,pq}A_{ij\alpha}A_{pq\alpha}-(\varphi^{-\frac{1}{k-l}} u^{-\frac{p-1}{k-l}})_{\alpha\alpha}-\varphi^{-\frac{1}{k-l}} u^{-\frac{p-1}{k-l}}+O(\eta)\notag\\
&=&-\sum\limits_{i,j,p,q\in\mathcal{G}}^{n}\widetilde{F}^{ij,pq}A_{ij\alpha}A_{pq\alpha}-(\varphi^{-\frac{1}{k-l}} u^{-\frac{p-1}{k-l}})_{\alpha\alpha}-\varphi^{-\frac{1}{k-l}} u^{-\frac{p-1}{k-l}}\notag\\
&&+O\left(\eta+\sum\limits_{i,j\in \mathcal{B}}|\n A_{ij}| \right). \label{keu-inequ-3}
\end{eqnarray}
From the inverse concavity of $\widetilde{F}(A)$,  cf. \cite[formula~(3.49)]{Urbas}, we obtain
\begin{eqnarray}\label{inverse conca}
    -\sum\limits_{i,j,p,q \in\mathcal{G}}\widetilde{F}^{ij,pq}A_{ij\alpha}A_{pq\alpha}-2\sum\limits_{i,j\in\mathcal{G}}\widetilde{F}^{ii}\frac{A_{ij\alpha}^{2}}{A_{jj}}\leq 0.
\end{eqnarray}
Substituting \eqref{inverse conca} and \eqref{keu-inequ-3} into \eqref{key-ine-2}, we get 
\begin{eqnarray}
    \sum\limits_{i,j=1}^{n}\widetilde{F}^{ij}\eta_{ij}
    &\leq&
    -\sum\limits_{\alpha\in\mathcal{B}}\left[\sigma_{m}(\mathcal{G})+\frac{\sigma_{1}^{2}(\mathcal{B}|\alpha)-\sigma_{2}(\mathcal{B}|\alpha)}{\sigma_{1}^{2}(\mathcal{B})}\right]\left[(\varphi^{-\frac{1}{k-l}} u^{-\frac{p-1}{k-l}})_{\alpha\alpha}+\varphi^{-\frac{1}{k-l}} u^{-\frac{p-1}{k-l}}\right]\notag \\
    &&+O\left(\eta+\sum\limits_{i,j\in\mathcal{B}}|\n A_{ij}|\right).\label{4.7}
\end{eqnarray}
Direct calculations yield
\begin{eqnarray}
   &&-(\varphi^{-\frac{1}{k-l}} u^{-\frac{p-1}{k-l}})_{\alpha\alpha}-\varphi^{-\frac{1}{k-l}} u^{-\frac{p-1}{k-l}} \notag \\
   &=&-\left[ \left(u^{-\frac{p-1}{k-l}}\right)_{\alpha\alpha}\varphi^{-\frac{1}{k-l}}+2 \left(u^{-\frac{p-1}{k-l}}\right)_{\alpha} \left(\varphi^{-\frac{1}{k-l}}\right)_{\alpha}+u^{-\frac{p-1}{k-l}}\left(\varphi^{-\frac{1}{k-l}}\right)_{\alpha\alpha}+u^{-\frac{p-1}{k-l}}\varphi^{-\frac{1}{k-l}}\right]\notag 
   \\
   &=&-\left[-\frac{p-1}{k-l}u^{-\frac{p-1}{k-l}-1}\varphi^{-\frac{1}{k-l}}(A_{\alpha\alpha}-u)+u^{-\frac{p-1}{k-l}}\left(\varphi^{-\frac{1}{k-l}} \right)_{\alpha\alpha}+u^{-\frac{p-1}{k-l}}\varphi^{-\frac{1}{k-l}}\right]\notag\\
   &&-\left[\frac{p-1}{k-l} \left(\frac{p-1}{k-l}+1\right)u^{-\frac{p-1}{k-l}-2}\varphi^{-\frac{1}{k-l}}u_{\alpha}^{2}-\frac{2(p-1)}{k-l}u^{-\frac{p-1}{k-l}-1}u_{\alpha}\left(\varphi^{-\frac{1}{k-l}}\right)_{\alpha}\right]\notag .
\end{eqnarray}
Combining \eqref{bad index}, $p\geq 1$ and the fact that $\varphi^{-\frac{1}{p+k-1-l}}$ is spherical convex in \eqref{varphi convex}, we get
\begin{eqnarray}
    &&-(\varphi^{-\frac{1}{k-l}} u^{-\frac{p-1}{k-l}})_{\alpha\alpha}-\varphi^{-\frac{1}{k-l}} u^{-\frac{p-1}{k-l}}\notag\\ 
    &\leq& -u^{-\frac{p-1}{k-l}}\left[ \left(\varphi^{-\frac{1}{k-l}} \right)_{\alpha\alpha}+\frac{p-1+k-l}{k-l}\varphi^{-\frac{1}{k-l}}+O(\eta)
  -\left(1-\frac{k-l}{p-1+k-l}\right)\frac{\left(\varphi^{-\frac{1}{k-l}}\right)_{\alpha}^{2}}{\varphi^{-\frac{1}{k-l}}}  \right]\notag\\
  &=&-u^{-\frac{p-1}{k-l}}\frac{p-1+k-l}{k-l}\varphi^{\frac{1}{p-1+k-l}-\frac{1}{k-l}}\left[ \left(\varphi^{-\frac{1}{p-1+k-l}}\right)_{\alpha\alpha}+\varphi^{-\frac{1}{p-1+k-l}}\right]+O(\eta)\notag \\
  &\leq& O(\eta). \label{key-inqu-4}
\end{eqnarray}
From \cite[Lemma~2.5]{BG}, 
\begin{eqnarray}\label{bad deri}
    \sum\limits_{i, j\in \mathcal{B}}|\n A_{ij}|\leq C\eta. 
\end{eqnarray}
Using  \eqref{key-inqu-4}, \eqref{4.7} and \eqref{bad deri}, we derive
\begin{eqnarray}\label{constant-rank}
  \sum\limits_{i,j=1}^{n}\widetilde{F}^{ij}\eta_{ij}\leq C(\eta+|\n \eta|),\quad \text{in}~\mathcal{O}.  
\end{eqnarray}
From the strong maximum principle, we have $\eta=0$ in $\mathcal{O}$, and since $\{x\in \mathbb{S}^{n}: \eta(x)=0\}$ is an open and closed set, then $\eta\equiv 0$ on $\SS^{n}$. This implies that $A$ has the constant rank $m$ on $\SS^{n}$. Now by the Minkowski formula (cf. \cite[Formula~(5.60)]{Sch})
\begin{eqnarray*}
    \int_{\SS^{n}}uH_{m}(A)d\mu=\int_{\SS^{n}}H_{m+1}(A)d\mu,
\end{eqnarray*}
we conclude that $u\equiv 0$ on $\SS^{n}$,  a contradiction. Hence $A=\n^{2}u+u\sigma>0$ on $\SS^{n}$.
\end{proof}
As an application of Theorem \ref{thm convex}, we now complete the proof of Theorem \ref{thm_new3}.
\begin{proof}[\textbf{Proof of Theorem \ref{thm_new3}}]\ 

    When $p>k-l+1$, we consider a one-parameter family of Equation  \eqref{t-equ} as in the proof of Theorem \ref{thm_new2}. It is easy to check that $\varphi_{t}$ satisfies condition \eqref{varphi convex} for any $t\in [0, 1]$.  If  there exists some $t_{0}\in(0,1]$ such that $u_{t}$ is strictly spherical convex for $0\leq t< t_{0}$, and $u_{t_{0}}$ is not strictly spherical convex, then by Theorem \ref{thm convex}, $u_{t_{0}}$ is strictly spherical convex, which is a contradiction. So we conclude that the solution $u$ is strictly spherical convex. When $p=k-l+1$, from the approximation process we know that $u$ is at least spherical convex. Then using Theorem \ref{thm convex} again,  we have that $u$ is strictly spherical convex.
    
\end{proof}

\section{Case $1<p<k-l+1$}\label{sec-4}
 In this section, we focus on the case $1<p<k-l+1$ and complete the proof of Theorem \ref{thm-new4}.  The crucial part of this section is to show uniform positive lower and upper bounds for even, strictly convex solutions to Equation \eqref{support-equ} when $1<p<k-l-1$. 

\subsection{A priori estimates}\ 

Following the idea in \cite[Lemma~2.2]{HI} (see also \cite[Proposition 2.1]{G}, \cite[Proposition 3.1]{GX} and \cite[Lemma 2.1]{HL}), we first establish a weighted gradient estimate for a solution to Equation  \eqref{support-equ}, which plays an important role in deriving the uniformly $C^{0}$ and $C^{1}$ estimates. Denote  $m_{u}:=\min\limits_{\SS^{n}}u$ and $M_{u}:=\max\limits_{\SS^{n}}u$. 
\begin{lem}\label{Lemma-weighted}
    Let $1<p<k-l+1$ and $1\leq l<k<n$. Suppose  $u$ is an even, strictly spherical convex solution to Equation  \eqref{support-equ}. Then there exist  positive constants $\beta\in \left(0, {2(p-1)}/{(k-l)}\right)$ and  $N$ depending on $n, k,l, p, \min\limits_{\SS^{n}}\varphi$, 
    and $\|\varphi\|_{C^{1}(\SS^{n})}$, such that
    \begin{eqnarray}\label{weighted-est}
        \frac{|\n u|^{2}+u^{2}}{u^{\beta}}\leq N \max_{\SS^n} u^{2-\beta}.
    \end{eqnarray}
\end{lem}

\begin{proof}
    We define the function 
    \begin{eqnarray*}
        \Phi:=\frac{|\n u|^{2}+u^{2}}{u^{\beta}},
    \end{eqnarray*}for some constant $\beta\in(0,2)$.
Suppose $\Phi$ attains its maximum value at some point, say $x_{0}\in\SS^{n}$. Then at $x_{0}$, we have
\begin{eqnarray*}\label{grad}
    0=\n \Phi=u^{-\beta}\left[(2\n^{2}u+2u\sigma)\n u-\beta u^{-1} (|\n u|^{2}+u^{2})\n u\right],
\end{eqnarray*}
then 
\begin{eqnarray}\label{one-deri}
    (\n^{2}u+u\sigma )\n u=\frac{\beta(|\n u|^{2}+u^{2})}{2u}\n u,
\end{eqnarray}
which implies that $\n u$ is an eigenvalue of $A$ at $x_{0}$.  Hence  we can choose an orthonormal frame at $x_{0}$, such that
\begin{eqnarray*}
    u_{1}=|\n u|,\quad \text{and}~\n^{2}u~\text{is~diagonal}.
\end{eqnarray*}
Denote $\rho:=\sqrt{|\n u|^{2}+u^{2}}$. Direct calculations yield
\begin{eqnarray*}
   \Phi_{ii}&=&\frac{2(u_{ii}^{2}+u_{s}u_{sii}+u_{i}^{2}+uu_{ii})}{u^{\beta}}-\frac{4\beta(u_{i}^{2}u_{ii}+uu_{i}^{2})}{u^{\beta+1}}-\beta\frac{\rho^{2}u_{ii}}{u^{\beta+1}}+(\beta+1)\beta\frac{\rho^{2}u_{i}^{2}}{u^{\beta+2}}\\
    &=&\Phi\left[\frac{2(A_{ii}^{2}+u_{s}A_{sii}-uA_{ii})}{\rho^{2}}-\frac{4\beta u_{i}^{2}A_{ii}}{u \rho^{2}}-\frac{\beta (A_{ii}-u)}{u}+\beta(\beta+1) \frac{u_{i}^{2}}{u^{2}}\right].
\end{eqnarray*}
By the maximal condition at $x_0$, we have 
\begin{eqnarray}
    0&\geq &\frac{1}{\Phi}\ov{F}^{ij}\Phi_{ij}\notag \\
    &=&\frac{2}{\rho^{2}}\ov{F}^{ii}(A_{ii}^{2}+u_{s}A_{sii})-\frac{2(k-l)\varphi u^{p}}{\rho^{2}}-\frac{4\beta \ov{F}^{11}u_{1}^{2}A_{11}}{u\rho ^{2}}-\beta (k-l)\varphi u^{p-2}\notag \\&&+\beta\sum\limits_{i=1}^{n}\ov{F}^{ii}+\beta(\beta+1)\frac{\ov{F}^{11}u_{1}^{2}}{u^{2}}.\label{nega}
    \end{eqnarray}
From \eqref{Ric-identity}, we see 
\begin{eqnarray}\label{codazzi}
A_{iis}=u_{iis}+u_{s}=u_{sii}+u_{i}\delta_{si}=A_{sii},
\end{eqnarray}
combining \eqref{one-deri}, \eqref{nega} and \eqref{codazzi}, we obtain
    \begin{eqnarray}
    0&\geq &\frac{2u_{s}}{\rho^{2}}\left(\varphi_{s}u^{p-1}+(p-1)\varphi u^{p-2}u_{s}\right)+\frac{2}{\rho^{2}}\ov{F}^{ii}A_{ii}^{2}-\frac{2(k-l)\varphi u^{p}}{\rho^{2}}\notag \\
    &&-\frac{2\beta^{2}\ov{F}^{11}u_{1}^{2}}{u^{2}}-\beta(k-l)\varphi u^{p-2}+\beta\sum\limits_{i=1}^{n}\ov{F}^{ii}+\beta(\beta+1)\frac{\ov{F}^{11}u_{1}^{2}}{u^{2}}\notag\\
    &=&\frac{2(p-1)\varphi u^{p-2}u_{1}^{2}}{\rho^{2}}+\frac{2u_{1}\varphi_{1}u^{p-1}}{\rho^{2}}-\frac{2(k-l)\varphi u^{p}}{\rho^{2}}-\beta(k-l)\varphi u^{p-2}+\beta\sum\limits_{i=1}^{n}\ov{F}^{ii}\notag \\
    &&+\frac{2}{\rho^{2}}\sum\limits_{i>1}\ov{F}^{ii}A_{ii}^{2}+\left(\frac{\beta^{2}\rho^{2}}{2u_1^{2}}-\beta(\beta-1)\right)\frac{\ov{F}^{11}u_{1}^{2}}{u^{2}}.\label{ine-3}
\end{eqnarray}
For $\beta\in(0,2)$, we have
\begin{eqnarray}\label{ine-1}
    \frac{\beta^{2}\rho^{2}}{2u_1^{2}}-\beta(\beta-1)\geq \frac{\beta^{2}}{2}-\beta(\beta-1)=\frac{\beta}{2}(2-\beta)> 0.\end{eqnarray}
 To proceed, we assume for some sufficiently large $N>2$, satisfying
\begin{eqnarray}\label{ine-2}
    u_{1}^{2}> N M_{u}^{2-\beta}u^{\beta}-u^{2}\geq \frac{N}{2}M_{u}^{2-\beta}u^{\beta},
\end{eqnarray}
 otherwise, we have completed the proof of  \eqref{weighted-est}. 
Substituting \eqref{ine-1}, \eqref{ine-2} into \eqref{ine-3}, we have
\begin{eqnarray*}
    0&\geq &\frac{2(p-1)\varphi u^{p-2}u_{1}^{2}}{\rho^{2}}+\frac{2u_{1}\varphi_{1}u^{p-1}}{\rho^{2}}-\frac{2(k-l)\varphi u^{p}}{\rho^{2}}-\beta(k-l)\varphi u^{p-2}\notag \\
    &=&\frac{2u^{p-2}\varphi}{\rho^{2}}\left[
    \left((p-1)-\frac{\beta(k-l)}{2}
    \right)u_{1}^{2}-uu_{1}|(\log\varphi)_{1}|-\left((k-l)+\frac{\beta(k-l)}{2}\right)u^{2} 
    \right]\\
    &\geq &\frac{2u^{p-2}\varphi}{\rho^{2}}
    \left[\frac{1}{2}\left((p-1)-\frac{\beta(k-l)}{2}\right)u_{1}^{2}-c_{0} u^{2}
    \right]
    \\
    &\geq &\frac{2u^{p-2}\varphi}{\rho^{2}}\left[\left((p-1)-\frac{\beta(k-l)}{2}\right)\frac{NM_{u}^{2-\beta}u^{\beta}}{4}-c_{0}u^{2}
    \right]\\
    &=&\frac{\varphi u^{p+\beta-2}N M_{u}^{2-\beta}}{2\rho^{2}}\left[\left((p-1)-\frac{\beta(k-l)}{2}\right)- 4c_{0}N^{-1} \left(\frac{u}{M_{u}}\right)^{2-\beta}\right],
\end{eqnarray*}
where the constant $c_{0}$ depends on $k,l, p$, and $\|\log \varphi\|_{C^{1}(\SS^{n})}$. 
If we choose $\beta\in \left(0, \frac{2(p-1)}{k-l}\right)$ and $N$ is sufficiently large such that 
$$(p-1)-\frac{\beta(k-l)}{2}-4c_{0}N^{-1}>0,$$
therefore we reach a contradiction. Hence we obtain \eqref{Lemma-weighted} and complete the proof.
\end{proof}

\begin{lem}\label{lemm-non-collosing}
    If $u$ is an even, strictly spherical convex function satisfying 
    \begin{eqnarray}\label{weighted-est-1}
        \frac{|\n u|^{2}}{u^{\beta}}\leq N \max_{\SS^n}{u}^{2-\beta}, \quad \text{on}~\SS^{n},
    \end{eqnarray}
    for some positive constants $\beta$ and $N$. Then the following non-collapsing estimate holds
    \begin{eqnarray*}
        \frac{\max\limits_{\SS^{n}}u}{\min\limits_{\SS^{n}}u}\leq C,
    \end{eqnarray*}
    where the constant $C$ depends on $n, \beta$ and $N$.
\end{lem}

\begin{proof} The proof is essentially the same as the result of Guan \cite[Lemma 3.1 and Corollary 3.1]{G}. 
For the reader's convenience, we include the proof here.

   Let $\Omega$ be the strictly convex body with the support function $u$. Since $u$ is an even, strictly spherical convex function, by John's Lemma, there exists an ellipsoid $E$ centered at the origin, such that
   \begin{eqnarray}\label{john}
       E\subset \Omega\subset n E.
   \end{eqnarray}
We write $E$ as
\begin{eqnarray*}
 \frac{x_{1}^{2}}{a_{1}^{2}}+\cdots+\frac{x_{n+1}^{2}}{a_{n+1}^{2}}=1,   
\end{eqnarray*}
with $a_{1}\geq a_{2}\cdots\geq  a_{n+1}>0$. Then we have 
\begin{eqnarray}\label{key-1}
    a_{1}\leq M_{u}\leq na_{1},\quad  m_{u}\leq a_{n+1} \leq n m_{u}.
\end{eqnarray}
The support function $u_{E}$ of the ellipsoid $E$ is given by
\begin{eqnarray*}
    u_{E}(x)=\left(a_{1}^{2}x_{1}^{2}+\cdots+a_{n+1}^{2}x_{n+1}^{2}\right)^{\frac{1}{2}},\quad x\in\SS^{n}.
\end{eqnarray*}
Hence by \eqref{john}, we have
\begin{eqnarray}\label{equ}
    u_{E}(x)\leq u(x)\leq nu_{E}(x),\quad \forall x\in\SS^{n}.
\end{eqnarray}
Restrict $u_{E}$ and $u$ on the slice $S:=\{x\in \SS^{n}: x=(x_{1}, 0, \cdots, 0, x_{n+1})\in\SS^{n}\}$.  First, we have
\begin{eqnarray}
    q(t):&=&u_{E}(t,0,\cdots, 0, \sqrt{1-t^{2}})\notag \\
    &=&(a_{1}^{2}t^{2}+(1-t^{2})a_{n+1}^{2})^{\frac{1}{2}}\geq a_{1}t, \quad \text{for}~t\in[0, 1].\notag
\end{eqnarray}
Then for any $t\in \left[0, \left(\frac{a_{1}}{a_{n+1}}\right)^{\frac{2-\beta}{2}}\right]$, 
\begin{eqnarray}\label{key-2}
    ta_{1}^{\frac{\beta}{2}}a_{n+1}^{\frac{2-\beta}{2}}\leq q\left(t \left(\frac{a_{n+1}}{a_{1}}\right)^{\frac{2-\beta}{2}}\right).
\end{eqnarray}
Denote 
\begin{eqnarray*}
    p(t):=\left[u \left(t,0,\cdots, 0, \sqrt{1-t^{2}}\right)\right]^{\frac{2-\beta}{2}}.
\end{eqnarray*}
From the weighted gradient estimate \eqref{weighted-est-1} and \eqref{key-1}, we get
\begin{eqnarray*}
   \left |\frac{d}{dt}p(t) \right|\leq N^{\frac{1}{2}}M_{u}^{1-\frac{\beta}{2}}\leq N^{\frac{1}{2}}(na_{1})^{\frac{2-\beta}{2}}.
\end{eqnarray*}
Using \eqref{equ}, we have
\begin{eqnarray*}
    p\left(t\left(\frac{a_{n+1}}{a_{1}}\right)^{\frac{2-\beta}{2}} \right)&\leq& p(0)+N^{\frac{1}{2}}(na_{1})^{\frac{2-\beta}{2}} t\left(\frac{a_{n+1}}{a_{1}} \right)^{\frac{2-\beta}{2}}\\
    &\leq &(n a_{n+1})^{\frac{2-\beta}{2}}+tN^{\frac{1}{2}}(n a_{n+1})^{\frac{2-\beta}{2}}.
\end{eqnarray*}
Therefore
\begin{eqnarray}\label{key-3}
    u\left(t \left(\frac{a_{n+1}}{a_{1}} \right)^{\frac{2-\beta}{2}}, 0,\cdots, 0, \left[1-t^{2} \left(\frac{a_{n+1}}{a_{1}}\right)^{2-\beta}\right]^{\frac{1}{2}}\right)\leq n
\left(1+tN^{\frac{1}{2}}\right)^{\frac{2}{2-\beta}}a_{n+1}.
\end{eqnarray}
Since $u(x)\geq u_{E}$, combining \eqref{key-2} and \eqref{key-3} we conclude that
\begin{eqnarray*}
    ta_{1}^{\frac{\beta}{2}}a_{n+1}^{\frac{2-\beta}{2}}\leq n \left(1+tN^{\frac{1}{2}}\right)^{\frac{2}{2-\beta}}a_{n+1}.
\end{eqnarray*}
Choosing $t=N^{-\frac{1}{2}}$,  it is clear that the above inequality implies 
\begin{eqnarray*}
    \frac{a_{1}}{a_{n+1}}\leq n^{\frac{2}{\beta}}N^{\frac{1}{\beta}}2^{\frac{4}{\beta(2-\beta)}},
    \end{eqnarray*}
together with \eqref{key-1}, we have
    \begin{eqnarray*}
        \frac{\max\limits_{\SS^{n}}u}{\min\limits_{\SS^{n}}u} 
        \leq n^{1+\frac{2}{\beta}}N^{\frac{1}{\beta}}2^{\frac{4}{\beta(2-\beta)}}.
    \end{eqnarray*}Hence we complete the proof.
\end{proof}

As a direct corollary of Lemma \ref{lemm-non-collosing}, we obtain uniform $C^{1}$ estimates for the solution $u$ to Equation  \eqref{support-equ}.
\begin{cor}\label{coro-C1}
      Let $1<p<k-l+1$ and $1\leq l<k<n$. Suppose  $u$ is an even,  positive and strictly spherical convex solution to Equation  \eqref{support-equ}. Then 
      \begin{eqnarray*}
          0<c<u<C, \quad \text{and}\quad |\n u|\leq C,
      \end{eqnarray*}
      where the constants $c, C$ depend on $n, k,l, p, \min\limits_{\SS^{n}}\varphi,$ 
      and $\|\varphi\|_{C^{1}(\SS^n)}$.
\end{cor}
\begin{proof}
    By the maximum principle, we have
    \begin{eqnarray*}
        (\max\limits_{\SS^{n}}u)^{p+l-(k+1)}\leq \max\limits_{\SS^{n}}\varphi^{-1}, \quad \text{and}\quad(\min\limits_{\SS^{n}}u)^{p+l-(k+1)}\geq \min\limits_{\SS^{n}}\varphi^{-1}.
    \end{eqnarray*}
    Since $p\in(1, k-l+1)$,  $p+l-(k+1)<0$. Hence we have a uniform upper bound for $\min\limits_{\SS^{n}}u$
    and a lower uniform positive bound for $\max\limits_{\SS^{n}}u$
    . Lemma \ref{Lemma-weighted} and Lemma \ref{lemm-non-collosing} imply
    that $0<c<u<C$ for some constants $c $ and $C$.  The gradient bound follows from Lemma \ref{Lemma-weighted}. 
\end{proof}
Combining Corollary \ref{coro-C1} and Lemma \ref{lem-C2}, we have the following theorem.
\begin{thm}\label{thm-convex est}
    Let $1<p<k+1-l$ and $1\leq l<k<n$. Let  $u$ be an even,  strictly spherical convex solution to Equation  \eqref{support-equ}. Then 
    \begin{eqnarray*}
        \|u\|_{C^{2}(\SS^n)}\leq C,
    \end{eqnarray*}
    where the constant $C$ depends on $n, k,l, p, \min\limits_{\SS^{n}}\varphi$, 
    and $\|\varphi\|_{C^{2}(\SS^n)}$.
\end{thm}

\subsection{Proof of Theorem \ref{thm-new4}}\

When $1<p<k+1-l$ and $1\leq l< k<n$, it becomes challenging to prove that the kernel of the corresponding linearization operator of Equation  \eqref{support-equ} is trivial. Therefore, we resort to the topological degree method, as introduced in \cite{LYY}. Following a similar argument as presented in \cite[Theorem 1.5]{GWR} or \cite[Section 3]{GG02}, we establish the existence of an even, strictly spherical convex solution to Equation  \eqref{support-equ}. Hence, we just provide an outline for the proof.

\begin{proof}[\textbf{Proof of Theorem \ref{thm-new4}}]\ 

Let us consider a one-parameter family of Equation  \eqref{t-equ} as in the proof of Theorem \ref{thm_new2}, and define the Banach space
\begin{eqnarray*}
    \mathcal{A}:=\{u\in C^{4}(\SS^{n}): u(x)=u(-x),\quad \text{for~all}~x\in \SS^{n}\}.
\end{eqnarray*}
It is obvious that $\varphi_{t}$ satisfies the assumption in Theorem \ref{thm convex}. This implies for any $0\leq t\leq 1$, the solution $u^{t}\in\mathcal{A}$ to Equation  \eqref{t-equ} is strictly spherical convex.  From Theorem \ref{thm-convex est}, we have the $C^{2}$ estimates, and the high-order estimates follow from the Evans-Krylov theorem. Hence we can use the degree theory. In order to use the degree theory in \cite{LYY} to establish the existence of the solution, we only need to compute the degree at $t=0$. By \cite[Theorem 6.4]{Udo} (or \cite[Theorem 1.3]{LWan24}), we have that $u=1$ is the unique solution for $\varphi^{0}=1$, and the same argument in \cite{GG02} yields that the degree is nonzero. Hence we have the existence of theorem \ref{thm-new4}. The convexity of the solution follows from  Theorem \ref{thm convex} under the assumption \eqref{varphi convex}.
\end{proof}


 \


\printbibliography 

\end{document}